\documentclass[runningheads]{llncs}
\usepackage[T1]{fontenc}
\usepackage{graphicx}
\usepackage{floatrow}
\usepackage[rflt]{floatflt}
\usepackage{float}
\usepackage{amsmath}
\usepackage{amsbsy}
\usepackage{amssymb}
\usepackage{textcomp}
\usepackage{graphicx}
\usepackage{amsfonts}
\usepackage{multirow}
\usepackage{url}
\usepackage{subfig}
\usepackage{longtable}
\usepackage{lscape}
\usepackage{hhline}
\usepackage[utf8]{inputenc}
\usepackage{colortbl,array} 
\usepackage{multirow,bigdelim}
\usepackage{booktabs}
\usepackage[linesnumbered,ruled]{algorithm2e}
\usepackage{xcolor,colortbl}
\usepackage[mathscr]{eucal}
\usepackage{mathtools, nccmath}
\usepackage{wrapfig}
\definecolor{Gray}{gray}{0.8}

\newcommand{\round}[1]{\ensuremath{\lfloor#1\rceil}}
\begin{document}
\title{Semi-Monotone Goldstein Line Search Strategy\\
 with Application in Sparse Recovery}
\titlerunning{Semi-Monotone Goldstein Line Search Strategy for Compressed Sensing}
%
\author{Shima Shabani\inst{}\orcidID{0009-0008-7052-2361} 
\and
Michael Breu{\ss}\inst{}
\orcidID{0000-0002-5322-2411}}
\authorrunning{S. Shabani and M. Breu{\ss}}
%
\institute{Institute for Mathematics, Brandenburg University of Technology\\Platz der Deutschen Einheit 1, 03046 Cottbus, Germany
\email{\{shima.shabani,breuss\}@b-tu.de}}
\maketitle              
\begin{abstract}
Line search methods are a prominent class of iterative methods to solve unconstrained minimization problems. 
These methods produce new iterates utilizing a suitable step size 
after determining proper directions for minimization. 
In this paper we propose a semi-monotone line search 
technique based on the Goldstein quotient for dealing with 
convex non-smooth optimization problems.
The method allows to employ large step sizes away from the optimum
thus improving the efficacy compared to standard Goldstein approach.
For the presented line search 
method, we prove global convergence to a stationary point and local 
R-linear convergence rate in strongly convex cases.
We report on some experiments in compressed sensing. 
By comparison with several state-of-the-art algorithms 
in the field, we demonstrate the competitive performance of 
the proposed approach and specifically its high efficiency.

\end{abstract}

\section{Introduction}

Line search methods are popular methods in unconstrained optimization \cite{Nocedal-Wright}.
They typically have the format of an iterative method combining a direction
of the iteration with a suitable step size.
Effective choices of both direction and step size are key to success and robustness
of the line search approach.
Given a suitable direction, monotone line search methods aim to ensure that a descent
in the objective function takes place, while non-monotone line search 
methods permit some growth in the function value, see e.g. \cite{GLL}. 
The non-monotone techniques represent a relaxation of monotone
line searching in a way that beneficial properties related 
to monotone methods like global convergence may be preserved,
while they help to improve the chance of 
finding a global optimum and in the convergence speed \cite{ZH}.

To select a practical step size, it is helpful to consider 
appropriate conditions to make sufficient decrease 
of the objective function. Popular conditions as discussed e.g. in
\cite{Nocedal-Wright} are the conditions proposed by
Armijo \cite{arm}, Wolfe \cite{WP}, or Goldstein \cite{GP}. 
As it is well-known, the Armijo condition
confirms the sufficient decrease for a given scheme even if the
step size is very small, which may result in the scheme being 
inefficient. The Wolfe condition overcomes this problem by
considering additionally the local curvature, but this comes 
at the expense of a considerably higher computational cost.
The Goldstein condition prevents too small step sizes
by making a possible lower bound explicit, which may 
eventually lead to higher 
computational efficiency compared to use of Wolfe condition.

As the setting for line search methods we consider in this paper
the unconstrained optimization problem
\begin{equation}\label{e:defprob0}
\begin{array}{ll}
\min &~F(x)=f(x)+\mu c(x)\\
\textrm{s.t.}&~ x \in \mathbb{R}^{n}
\end{array}
\end{equation}
where the function $f(x)$ is convex and smooth, the regularization
$c(x)$ is a convex and non-smooth function, and the parameter
$\mu\in\mathbb{R^{+}}$ realizes a trade-off specification between the
objective terms.
One particular form of \eqref{e:defprob0} that is of high interest 
in some contexts is the \textit{basis pursuit denoising}
problem \cite{Do,EK,FR}, denoted as
\begin{equation}\label{e:defprob1}
\begin{array}{ll}
\min &~F(x)=\frac{1}{2}\|Ax-b\|^2+\mu\|x\|_1 \\
\textrm{s.t.}&~ x \in \mathbb{R}^{n} 
\end{array}
\end{equation}
Here $\|\cdot\|$ is the standard Euclidean norm,
$\|\cdot\|_1$ stands for the $\ell_1$-norm, and 
$A\in\mathbb{R}^{m\times n}$ $(m\ll n)$ and $b\in\mathbb{R}^m$ are 
often called the measurement matrix and the observation vector, respectively. 
The model problem \eqref{e:defprob1} gives the sparsest solution of the 
underdetermined linear system $Ax=b$ when the measurements
are noisy, meaning that it is desired that almost all the coefficients 
of the linear combination form are zero. 
Compressed sensing \cite{EK,FR}, a technique for efficiently acquiring 
and reconstructing a signal, is an actual field of application for 
\eqref{e:defprob1}. 



{\bf Our contribution and relation to previous work.}
We build upon recent work on
optimization methods relying on the Goldstein 
quotient \cite{Go}.
The works \cite{NM,NMA} present an improved version of 
the Goldstein technique similarly to our underlying
approach, but limited to smooth optimization problems. 
With respect to non-smooth optimization, 
the most important work related to this paper
is given by \cite{SB}. 
There an efficient method utilizing a monotone 
Goldstein technique has been deviced for dealing with the 
basis pursuit denoising problem, proving global and 
R-linear local convergence in that setting.

In the current paper, we propose a new extension of 
technique proposed in \cite{SB}
that enables relaxation towards larger step sizes during 
minimization.
As the novel technique keeps the lower bound of the monotone
version of Goldstein's condition at the same time, we denote it 
here as semi-monotone Goldstein strategy.
For theoretical validation we prove global and 
R-linear local convergence of the resulting line search method
in the setting of a general convex non-smooth regularization 
function as in \eqref{e:defprob0}.
This is a considerably more general setting as in 
previous work \cite{SB}, as convergence proofs given 
there rely on a monotone Goldstein technique
and are also limited to the shrinkage operation 
for regularization as in \eqref{e:defprob1}.
In addition to our new theoretical results we present
a detailed evaluation of randomized numerical experiments 
that show the beneficial properties of our approach
compared to the scheme proposed in \cite{SB}
as well as its competitive performance with 
several state-of-the-art optimization methods.
\section{Review of Important Techniques}\label{subsec:con}

We consider a general iterative method that seeks 
the minimum of \eqref{e:defprob0} and produces 
to this end a sequence $\{x_k\}_{k\geq 0}$ by the 
line search rule.
In the following, the components of 
line search iteration and related concepts 
are briefly summarized for later use.

The basic line search scheme at an iteration index $k$ reads as
\begin{equation}\label{e:iterscheme}
x_{k+1} \, = \, x_k+\alpha_k d_k
\end{equation}
and it features the direction $d_{k}$ and a positive step size $\alpha_{k}$. 
For \eqref{e:defprob0}, the descent in the objective function in
direction $d_k$ from current iterate $x_k$ can be measured by
\begin{equation}\label{e:desdks0}
\Delta_k=d_k^T\nabla f_k+\mu
\Big(c(x_k+d_k)- c_k\Big)<0
\end{equation}
As for the notation in \eqref{e:desdks0}, for a function like $h(x)$, 
$h_k$ is the function value at $x_ k$ or $k$th iterate, and 
$\nabla h_k$ is the corresponding gradient value. 
The Goldstein condition for a non-smooth optimization strategy 
can be summarized, using $0<\theta_1<\theta_2<1$, as
\begin{equation}\label{e:gol2}
F_k+\theta_2\alpha_k\Delta_k\leq F_{k+1}\leq
F_k+\theta_1\alpha_k\Delta_k
\end{equation}
In \cite{SB} an efficient monotone 
line search technique for \eqref{e:defprob0} is presented
in the sense of {\sc Warth \& Werner} \cite{WW0}.
This means, in each iteration, the step size $\alpha_k$ 
is chosen such that the following quotient has a fixed 
positive lower bound $\gamma$ 
\begin{equation}\label{e:warth}
\Big(
F_k-F_{k+1}\Big)\frac{\|d_k\|^2}{\Delta_k^2}>\gamma>0
\end{equation}
for all $k\geq 0$ with descent direction $d_k$ as
\begin{equation}\label{e:dpro}
d_{k}=x_k^+-x_k,~~~~x_k^+=\mathscr{P}_{c}(x_{k}-\tau_k\nabla
f(x_k), \mu \tau_k)
\end{equation}
Here 
\begin{equation}\label{e:dpro1}
\mathscr{P}_{c}(x,\vartheta)=\textrm{argmin}_{y}\Big\{\vartheta c(y)+\frac{1}{2}\|x-
y\|^2\Big\}
\end{equation}
is the proximal operator \cite{PB} of the scaled function $c(y)$ at $x$ with $\vartheta>0$, 
and $\tau_k$ is a dynamic parameter employed as a step size for the implicit gradient descent step such that 
\begin{equation}\label{e:tauk}
 0<\tau_{\min}\leq\tau_k\leq\tau_{\max}<\infty   
\end{equation}
To estimate $\alpha_k$ in \eqref{e:iterscheme}, the authors 
of \cite{SB} consider the non-smooth Goldstein quotient
\begin{equation}\label{e:gol1}
\nu(\alpha_k)=\frac{F_{k+1}-F_k}{\alpha_k\Delta_k}
\end{equation}
by imposing the following monotone sufficient descent condition
\begin{equation}\label{e:sufdes}
\nu(\alpha_k)\left|\nu(\alpha_k)-1\right| \, \geq \, \theta
\end{equation}
for a fixed $\theta>0$. 
%
Based on \eqref{e:sufdes}, $\nu(\alpha_k)$ is supposed to be 
sufficiently positive and far from one to control 
the step size from below.

\section{Semi-Monotone Goldstein Strategy}\label{smgs}

As indicated, the idea presented in this paper extends the monotone 
line search method in \cite{SB} to a semi-monotone case, 
to use a larger board compared with \eqref{e:gol2} for 
estimating $\alpha_k$ and simultaneously controlling it from below. 

Let us define $l(k)$ as an integer satisfying $k-m(k)\leq l(k) \leq k$, in which $m(0) = 0$ and $0 \leq m(k)\leq \min\{m(k-1)+1, N-1\}$, with $N > 0$, to identify after a starting phase the maximum function value of $N$ prior successful iterates in step $k$ with corresponding index $l(k)$ via
\begin{equation}\label{e:nm2}
F_{l(k)}=\max\limits_{0\leq j \leq m(k)}\{F_{k-j}\}
\end{equation}
We consider the convex combination of $F_{l(k)}$ and current 
function value $F_k$ as 
\begin{equation}\label{e:nm1}
  R_k=\eta_k F_{l(k)}+(1-\eta_k)F_k 
\end{equation}
where $R_k\geq F_k$ in this way, and we have
$\eta_k\in[\eta_{\min},\eta_{\max}]$, $\eta_{\min}\in[0,1]$, and $\eta_{\max}\in[\eta_{\min},1]$.

As indicated, our aim is to keep the lower bound in \eqref{e:gol2} 
but to relax the upper bound to tackle the non-monotone case. 
To realize this we consider the estimate from \eqref{e:nm1} to obtain
\begin{equation}\label{e:golsnm}
F_k+\theta_2\alpha_k\Delta_k\leq F_{k+1}\leq
R_k+\theta_1\alpha_k\Delta_k
\end{equation}
Now, estimating $\alpha_k$, we define the relaxed version 
of the non-smooth Goldstein quotient \eqref{e:gol1} as 
\begin{equation}\label{e:golnm1}
\lambda(\alpha_k)=\frac{F_{k+1}-R_k}{\alpha_k\Delta_k}
\end{equation}
Motivated by \eqref{e:golsnm}, we impose the following 
\emph{semi-monotone sufficient descent} condition    
\begin{equation}\label{e:sufdescom}
\nu(\alpha_k)\left|1-\lambda(\alpha_k)\right| \, \geq \, \theta
\end{equation}
using \eqref{e:gol1} and \eqref{e:golnm1} to control the size 
of $\alpha_k$, especially keeping it away from zero.

We will show that condition \eqref{e:sufdescom} guarantees a reasonable decrease for the objective function.
Regarding \eqref{e:nm1}, by adopting the value for $\eta_k$, \eqref{e:golnm1} allows to employ a non-monotone technique whenever iterates are far away from the optimum and a nearly monotone whenever iterates are close to the optimal value. 

\section{Convergence Analysis}

In this section we discuss convergence properties of the presented line search method. Investigating convergence analysis, we assume 
$\nabla f(x)$ is Lipschitz continuous with constant $L>0$, i.e.,
for all $x,y \in \mathbb{R}^n$
\begin{equation}\label{e:lip}
    \|\nabla f(x)-\nabla f(y)\|\leq L \|x-y\|
\end{equation}
or equivalently
\begin{equation}\label{e:lip1}
   f(y) \leq f(x)+\nabla f(x)^T(y-x)+\frac{L}{2} \|x-y\|^2
\end{equation}

\subsection{Bounded Gain in the Target Function}

The first theorem implies \emph{efficiency} of the iterative 
method that enforces a bounded gain to the target function 
as described in \eqref{e:warth}. 

\begin{theorem}\label{thm:1} If for the descent direction $d_k$, 
the value $\alpha_k$ satisfies the sufficient descent condition \eqref{e:sufdescom} and in addition
\begin{equation}\label{e:ass0}
F_{k+1}-R_{k}\geq -\frac{\alpha_k^2}{2}L\|d_k\|^2+\alpha_k\Delta_k
\end{equation}
is satisfied, then for any step size $\alpha_k'$ with
$F(x_k+\alpha_k' d_k)\leq F_{k+1}$ we have
\begin{equation}\label{e:eq1}
\Big( F_k-F(x_k+\alpha_k'
d_k)\Big)\frac{\|d_k\|^2}{\Delta_k^2}\geq\frac{2\theta}{L}
\end{equation}
\end{theorem}
\begin{proof} Convexity of $c(x)$ in \eqref{e:defprob0} 
results in 
\begin{equation}
\begin{split}\label{e:mainineq0}
 &-\frac{\alpha_k^2}{2}L\|d_k\|^2+\alpha_k \Delta_k\overset{\eqref{e:ass0}}{\leq}~F_{k+1}-R_k~
\overset{\eqref{e:nm1}}{\leq} F_{k+1}-F_k\\
 &\overset{\text{\eqref{e:defprob0}}}{=}f_{k+1}-f_k+\mu\big(c_{k+1}-c_k\big)\overset{\eqref{e:lip1}}{\leq}\alpha_k
d_k^T\nabla f_k+
\frac{\alpha_k^2}{2}L\|d_k\|^2+\mu\big(c_{k+1}-c_k\big)\\
&\underset{\text{c(x)}}{\overset{\text{Conv}}{\leq}}\frac{\alpha_k^2}{2}L\|d_k\|^2+\alpha_k
d_k^T\nabla f_k+
\mu\alpha_k\Big(c(x_k+d_k)-c(x_k)\Big)
\overset{\text{\eqref{e:desdks0}}}{=}\frac{\alpha_k^2}{2}L\|d_k\|^2+\alpha_k
\Delta_k   
\end{split}
\end{equation}
Therefore, in total we can obtain
\begin{equation}\label{e:mainineq1}
\left|F_{k+1}-R_k-\alpha_k \Delta_k\right|\leq\frac{\alpha_k^2}{2}L\|d_k\|^2
\end{equation}
Making then use of \eqref{e:golnm1} gives
\begin{equation}\label{e:th11}
\frac{\|d_k\|^2}{\left|\Delta_k\right|}\geq\frac{2}{\alpha_k
L}\left| \frac{F_{k+1}-R_k-\alpha_k
\Delta_k}{\alpha_k \Delta_k}\right|
=\frac{2}{\alpha_k
L}\left|\lambda(\alpha_k)-1\right|
\end{equation}
We have $\Delta_k<0$ in \eqref{e:desdks0}. Therefore by 
definition of $\nu(\alpha_k)$ from \eqref{e:gol1}
\begin{equation}\label{e:th12}
\frac{F_k-F_{k+1}}{\left|\Delta_k\right|}=\frac{F_{k+1}-F_k}{\Delta_k}=\alpha_k\nu(\alpha_k)
\end{equation}
Taking the product of terms from \eqref{e:th11} and 
\eqref{e:th12} and using \eqref{e:sufdescom} leads to the result. \qed
\end{proof}

\subsection{Reaching to a Stationary Point}

The following theorem, the proof of which can be transferred 
identically from \cite{SB} to the current setting, 
characterizes the situation to reach an arbitrarily 
negative function value respectively getting arbitrarily close 
to a stationary point.

\begin{remark}\label{rem1}
Based on Lemma 3.1 and Lemma 2.1 in \cite{TY} for $d_k$ by \eqref{e:dpro} 
\begin{itemize}
    \item [i.] $x_k$ is a stationary point for problem \eqref{e:defprob0} if and only if
$d_k=0$.
    \item[ii.] 
$d_k\neq 0$ with $\tau_k$ by \eqref{e:tauk} is descent i.e., in \eqref{e:desdks0}
\begin{equation}\label{e:desdks}
\Delta_k\leq-\frac{1}{2\tau_k}\|d_k\|^2
\end{equation}
\end{itemize}
\end{remark}
\begin{theorem}\label{thm:3}
Let $\Tilde{\Delta_k}=\left|\Delta_k\right|>0$. Suppose that for every
integer $k\geq1$
\begin{equation}\label{e:th32}
\sup_k \|d_k\|^2\Big(
\frac{\|d_k\|^2}{\Tilde{\Delta}_k^2}-\frac{\|d_{k-1}\|^2}{\Tilde{\Delta}_{k-1}^2}\Big)<\infty
\end{equation}
Then
\begin{equation}\label{e:th34}
\lim_{k\rightarrow{\infty}} \|d_k\|=0~~~\textrm{or}~~~\lim_{k\rightarrow{\infty}}F_k=-\infty
\end{equation}
\end{theorem}
Discarding the negative term in \eqref{e:th32} and considering \eqref{e:tauk} and \eqref{e:desdks} we have 
\begin{equation}\label{e:sup}
\frac{\|d_k\|^4}{\Tilde{\Delta}_k^2}\leq 4\tau_{\max}^2   
\end{equation}
which leads to \eqref{e:th32} with $\Tilde{\Delta_k}>0$.

\subsection{R-Linear Local Convergence}
To investigate R-linear convergence, we first present some preliminaries. Notice that one can reformulate \eqref{e:defprob0} as a smooth minimization problem over a closed convex set or epigraph of $c(x)$ \cite{TY} as 
\begin{equation}\label{e:refprob0}
\begin{array}{ll}
\min &~F(x)=f(x)+\mu \zeta\\
\textrm{s.t.}&~ (x,\zeta) \in \Omega=\big\{(x,\zeta)\mid c(x)\leq\zeta\big\}
\end{array}
\end{equation}
So, the first optimality condition for a stationary point $x^*$ is
\begin{equation}\label{e:optcon}
\nabla f(x^*)^T(x-x^*)+\mu\big(\zeta-c(x^*)\big)\geq 0, ~~~\forall (x,\zeta) \in \Omega 
\end{equation}
since, $\zeta^*=c(x^*)$.
We consider the generalized form of Lemma 2.1 in \cite{WYZG} as 
\begin{lemma}\label{pr:pro1}
For the proximal operator \eqref{e:dpro}, we have the following properties for all $x, y \in \mathbb{R}^n$ and $\vartheta>0$:
\begin{itemize}
\item[P1.]$\big(\mathscr{P}_{c}(x,\vartheta)-x\big)^T\big(y-\mathscr{P}_{c}(x,\vartheta)\big)+\vartheta\big(\zeta-c(\mathscr{P}_{c}(x,\vartheta))\big)\geq 0$, $\forall$  $(y,\zeta)\in \Omega$
\item[P2.] $\big(\mathscr{P}_{c}(x,\vartheta)-\mathscr{P}_{c}(y,\vartheta)\big)^T\big(x-y\big)\geq \|\mathscr{P}_{c}(x,\vartheta)-\mathscr{P}_{c}(y,\vartheta)\|^2$
\item[P3.] $\|\mathscr{P}_{c}(x,\vartheta)-\mathscr{P}_{c}(y,\vartheta)\|\leq\|x-y\|$
\end{itemize}
\end{lemma}
\begin{proof}
    \begin{itemize}
        \item [\textit{P1.}] Considering unconstrained nonsmooth minimization problem in \eqref{e:dpro1} as a constrained smooth version in the way \eqref{e:refprob0} and applying \eqref{e:optcon} to its target function with $f(y)=\frac{1}{2}\|x-y\|^2$, $\nabla f(y)=y-x$ and $x^*=\mathscr{P}_{c}(x,\vartheta)$ gives the result.
        \item[\textit{P2.}] Replacing $y$ with $\mathscr{P}_{c}(y,\vartheta)$ and $\zeta$ with $c\big(\mathscr{P}_{c}(y,\vartheta)\big)$ in \textit{P1} gives
        \begin{equation}
         \big(\mathscr{P}_{c}(x,\vartheta)-x\big)^T\big(\mathscr{P}_{c}(y,\vartheta)-\mathscr{P}_{c}(x,\vartheta)\big)+\vartheta \big(c\big(\mathscr{P}_{c}(y,\vartheta)\big)-c\big(\mathscr{P}_{c}(x,\vartheta)\big)\big)\geq 0   
        \end{equation}
        We obtain a similar inequality if $x$ and $y$ is exchanged. Adding these two inequalities leads to the result.
        \item[\textit{P3.}] Apply the Cauchy-Schwarz inequality (CSI)
        to the left hand side of \textit{P2}.
    \end{itemize}\qed
\end{proof}
Let us note that we did not find the explicit proof of \textit{P1} in the literature. The standard proof of \textit{P2} in the literature relies on the use of subgradient optimality condition, see e.g. \cite{PB}. Here, we employ directly \textit{P1}. The proof of \textit{P3} is standard and recalled here for convenience. 
\begin{remark}\label{r:rem2}
f(x) is strongly convex with constant $L'>0$ if for all $x,y \in \mathbb{R}^n$
\begin{equation}\label{e:strcon}
    \big(\nabla f(x)-\nabla f(y)\big)^T(x-y)\geq L' \|x-y\|^2 
\end{equation}
\end{remark}
\begin{lemma}\label{l:lem2}
Suppose $x^*$ is a stationary point for \eqref{e:defprob0}. If $f(x)$ is strongly convex with $L'>0$ and $\nabla f(x)$ is Lipschitz continuous with $L>0$, then for every $k$, 
\begin{equation}\label{e:stcon}
 \|x_k-x^*\|\leq\frac{1+\tau_k{L}}{
\tau_k{L'}}\|d_k\|   
\end{equation}
\end{lemma}
\begin{proof}
Replacing $x$ with $x_k-\tau_k \nabla f_k$, $y$ with $x^*$, $\zeta$ with $c(x^*)$ and $\vartheta=\mu\tau_k$ in \textit{P1} 
\begin{equation}
\begin{split}
\big(\mathscr{P}_{c}(x_k-\tau_k\nabla f_k,\mu\tau_k)-x_k+\tau_k\nabla f_k\big)^T\big(x^*-\mathscr{P}_{c}(x_k-\tau_k\nabla f_k,\mu\tau_k)\big)\\+\mu\tau_k\big(c(x^*)-c(\mathscr{P}_{c}(x_k-\tau_k\nabla f_k,\mu\tau_k))\big)\geq 0  
\end{split}
\end{equation}
which by \eqref{e:dpro} is equivalent to
\begin{equation}\label{e:p1} 
\big(d_k+\tau_k \nabla f_k\big)^T\big(x^*-x_k^+\big)+\mu\tau_k\big(c(x^*)-c(x_k^+)\big)\geq 0  
\end{equation}
Since $\big(x_k^+,c(x_k^+)\big)\in \Omega$, the optimality condition \eqref{e:optcon} gives
\begin{equation}\label{e:opt} 
\tau_k \nabla f(x^*)^T\big(x_k^+-x^*\big)\geq\mu\tau_k\big(c(x^*)-c(x_k^+)\big) 
\end{equation}
which by \eqref{e:p1} gives
\begin{equation}\label{e:optres} 
\big(d_k+\tau_k (\nabla f_k-\nabla f(x^*)\big)^T\big(x^*-x_k^+\big)\geq 0  
\end{equation}
Expanding \eqref{e:optres} by using \eqref{e:dpro} and rearranging terms, we obtain
\begin{equation}\label{e:optres1} 
 d_k^T(x^*-x_k)-\|d_k\|^2+\tau_kd_k^T(\nabla f(x^*)-\nabla f_k)\geq \tau_k(\nabla f_k-\nabla f(x^*))^T(x_k-x^*)   
\end{equation}
Defining $\delta x_k^* := x_k -  x^*$ 
and employing the CSI we have 

\begin{equation}
\begin{split}
 &\tau_k L'\|\delta x_k^*\|^2\overset{\eqref{e:strcon}}{\leq}\tau_k(\nabla f_k-\nabla f(x^*))^T  \delta x_k^*
 \overset{\text{\eqref{e:optres1}}}{\leq}d_k^T\delta x_k^*
 +\tau_kd_k^T(\nabla f(x^*)-\nabla f_k)\\
&{\overset{\text{CSI}}{\leq}}\|d_k\|\|\delta x_k^*\|
+\tau_k \|d_k\|\|\nabla f(x^*)-\nabla f_k\|
\overset{\text{\eqref{e:lip}}}{\leq}(1+\tau_kL)\|d_k\|\|\delta x_k^*\|
\end{split}
\end{equation}
which leads to the result.\qed
\end{proof}
\begin{remark}\label{r:rem3}
    According to Theorem \ref{thm:1}, the relations \eqref{e:tauk}, and \eqref{e:desdks}, we have 

\begin{equation}\label{e:strinf}  
\frac{F_k-F_{k+1}}{\|d_k\|^2}=\Big(F_k-F_{k+1}\Big)\frac{\|d_k\|^2}{\Delta_k^2}\Big(\frac{\Delta_k}{\|d_k\|^2}\Big)^2\underset{\text{Th1}}{\overset{\eqref{e:desdks}}{\geq}}\frac{\theta}{2L\tau_{\max}^2}\overset{\text{\eqref{e:tauk}}}{\geq}\frac{\theta}{2L\tau_{\max}^2}>0
\end{equation}
\end{remark}
\begin{theorem}\label{thm:4}
Suppose that $d_k$ is produced by\
\eqref{e:dpro}, the sequence $\{x_k\}$ converges to a minimizer $x^*$, and suppose $f(x)$ is a strongly convex function, then there are constants $q\in(0,1)$ and $\xi_1, \xi_2, \xi_3>0$
such that, for all $k$, we have
\begin{equation}\label{e:th42}
F_{k}-F(x^*)\leq \xi_1 q^{2k},~~~~\|x_k-x^*\|\leq \xi_2
q^{k},~~~~\|d_k\|\leq \xi_3 q^{k}
\end{equation}
\end{theorem}
\begin{proof}
Using \textit{P3} of Lemma \ref{pr:pro1}, Lemma \ref{l:lem2}, and Remark \ref{r:rem3}, the corresponding proof of Theorem 3 in \cite{SB} can be transferred identically to our setting. \qed
\end{proof}


\section{Numerical Experiments}

In order to give a reasonable comparison to the scheme
proposed in \cite{SB} and explore the usefulness 
of the approach in compressed
sensing applications, we focus on the model
problem \eqref{e:defprob1} in our experiments. 

To this end we employ $c(x)=\|x\|_1$ in \eqref{e:defprob0}, so that $d_k$ in \eqref{e:dpro} is the \textit{shrinkage operator}\ \cite{PB,SB}
\begin{equation}\label{e:dks}
d_k=\mathcal{S}\Big(x_k-\tau_k \nabla f(x_k),\mu \tau_k\Big)-x_k
\end{equation}
Using \eqref{e:dks} with dynamic parameter $\tau_k$ in \eqref{e:tauk} and the proposed semi-monotone technique in Section \ref{smgs}, we obtain the new \emph{semi-monotone iterative shrinkage Goldstein algorithm} ({\tt smISGA}). 

Numerical experiments are performed to demonstrate the properties of {\tt smISGA} and to compare it to the compressed sensing reconstruction algorithms {\tt FPC-BB} \cite{HYZ}, {\tt TwIST} \cite{BF}, {\tt FISTA} \cite{BT}, {\tt SpaRSA} \cite{WNF}, and {\tt ISGA} \cite{SB}. The experiments  were run in MATLAB R2024b on an egino BTO with 128 × Intel$^{\circledR}$ Xeon$^{\circledR}$ Gold 6430 and 125,1 GiB of RAM.  

\subsubsection{Experimental Setting}
By specifying size and type of the matrix, problem data $A$ and $b$ are generated in \eqref{e:defprob1}. We employ the same methodology for that as in \cite{SB}, which we now briefly recall.
Based on randomization of entries or composition, 
six types of test matrices $A$ representing 
standard tests in 
compressed sensing are evaluated: 
{\tt{Gaussian}}, {\tt{scaled Gaussian}},
{\tt{orthogonalized Gaussian}}, {\tt{Bernoulli}}, 
{\tt{partial Hadamard}}, and 
{\tt{partial discrete cosine transform}}.

Given the dimension of the signal $x$ as $n\in\{2^{10},\ldots,
2^{15}\}$, in agreement with the $\rho$ and 
$\delta$ in $\{0.1,0.2, 0.3\}$, the dimension of the observation vector $b$ 
is produced by $m=\round{\delta n}$, and  
the number of nonzero elements in an exact 
solution $xs$ is given by $k=\round{\rho m}$.
Where $\round{\cdot}$ indicates rounding
to closest integer.

For statistical evaluation, $xs$ and $b$ were contaminated 
by Gaussian noise with values in
$\{(10^{-h},10^{-h}) \, | \, h=1, 3, 5, 7 \}$.
This leads to the generation 
of $1296$ random test problems in
the compressed sensing field. 

\begin{figure}[ht]
\begin{center}
\begin{tabular}{llll}
{\includegraphics[width=4.0cm]{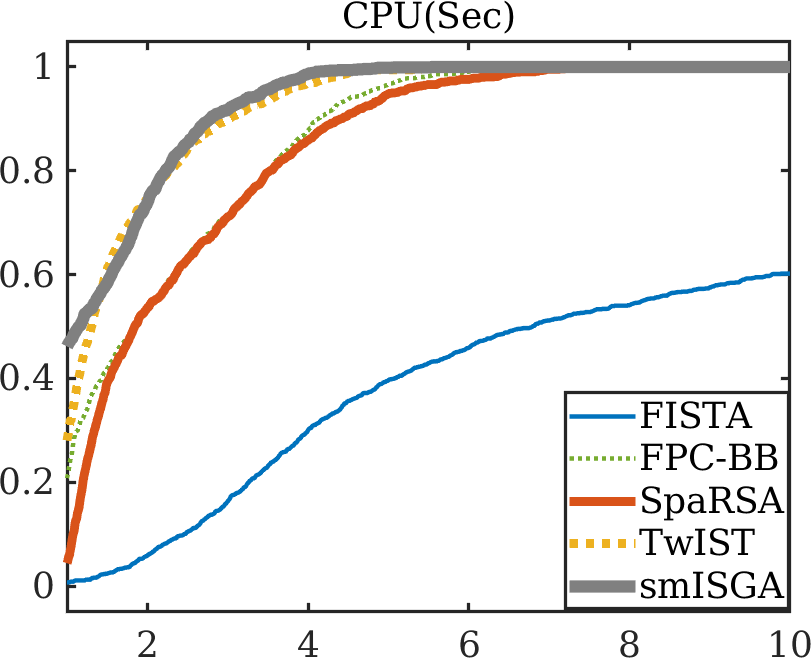}}{\includegraphics[width=4.0cm]{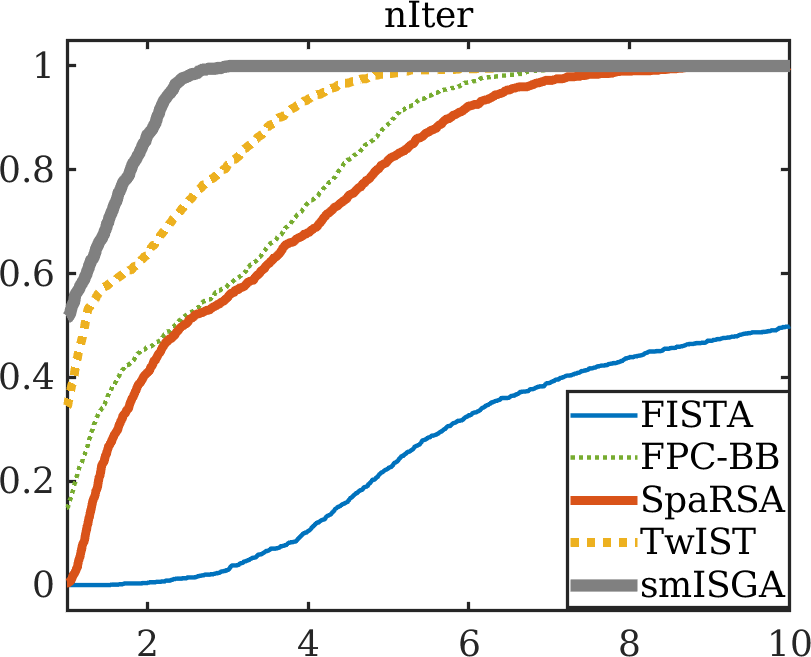}}{\includegraphics[width=4.0cm]{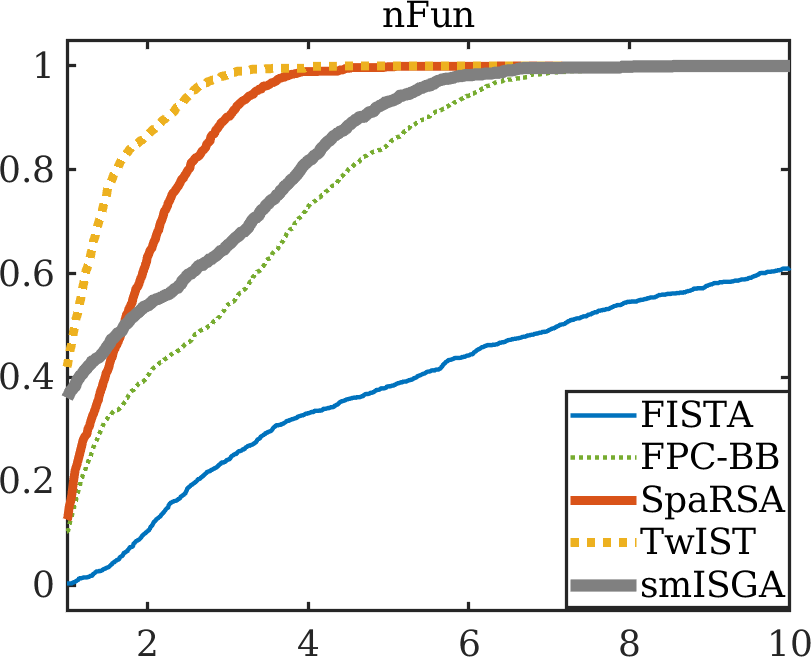}}\\
\end{tabular}
\caption{\label{Fig:P1}
Comparison among {\tt FISTA}, {\tt FPC-BB}, {\tt
SpaRSA}, {\tt TwIST}, and {\tt smISGA} with the performance cost
metric {\tt CPU(Sec)}, {\tt nIter}, and {\tt nFun}, respectively. $x$ and $y$ axes indicate $\varsigma$ and $P_{{\tt s}}(\varsigma)$, respectively; see \eqref{e:perfo}.
}
\end{center}
\end{figure}

\subsubsection{Implementation Details}
For all solvers, the initial point value 
$x_0$ is set to the zero vector, 
and we set $\mu=2^{-8}$. 
For {\tt FPC-BB}, {\tt SpaRSA}, {\tt ISGA} and {\tt smISGA}, 
we select safeguard parameters as
$\tau_{\min}=10^{-4}$, $\tau_{\max}=10^4$, 
as well as $\theta= 10^{-10}$ 
specifically for {\tt ISGA} and {\tt smISGA}. 
The stopping condition for all solvers is given by 
\begin{equation}
  \|F(x_{k+1})-F(x_k)\|\leq {\tt {Ftol}}\|F(x_k)\| 
\end{equation}
with ${\tt {Ftol}}=10^{-10}$, or if the number of iterations exceeds
10000.
According to \cite{ESM}, the parameter $\eta_k$, with $\eta_0=0.5$, is updated by 
$\eta_k = \frac{2}{3}\eta_{k-1}+0.01$
if $\|\nabla f_k\| \leq 10^{-2}$, and 
$\eta_k = \max\{0.99\eta_{k-1}, 0.5\}$ else.

\subsubsection{Evaluation of Computational Efficiency}
Our cost metrics are here
the time in seconds ({\tt CPU(Sec)}), the total number of iterations 
({\tt nIter}), and the number of function evaluations ({\tt nFun}).
We first consider \emph{performance profiles} \cite{DM2} 
to benchmark and compare solvers on the generated test
problems. 

Let us give a short explanation of the performance profiles. 
We consider a number of $\tt {n_s}$ solvers in a set of solvers 
$\tt S$ and $\tt {n_p}$ problems in the problem set $\tt P$. 
We want for example to evaluate 
computing time via the performance profile. 
Let $t_{{\tt p},{\tt s}}$ thus be computing time 
to solve problem $\tt p$ by solver $\tt s$. 
We compare the performance on $\tt p$ by $\tt s$ with the best performance 
by any solver tackling $\tt p$ via the quantity 
$ r_{{\tt p},{\tt s}}$ that is defined as follows:
\begin{equation}\label{e:perfo}
    r_{{\tt p},{\tt s}}=\frac{t_{{\tt p},{\tt s}}}{\min \{t_{{\tt p},{\tt s}}:{\tt s}\in {\tt S}\}},
    ~~~~P_{{\tt s}}(\varsigma)=
    \frac{\displaystyle{
    \mathrm{size}\{{\tt p}\in{\tt P}:r_{{\tt p},{\tt s}}\leq \varsigma\}}}{n_{p}}
\end{equation}
Then in \eqref{e:perfo}, $P_{{\tt s}}(\varsigma)$ gives the probability for $\tt s$ that a performance ratio $r_{{\tt p},{\tt s}}$ is within a factor $\varsigma\in\mathbb{R}$. 
\begin{table}[ht]
\caption{\label{table-sb-1}
Average values (left block) and their standard deviations (right block)
of the cost metrics {\tt CPU(Sec)}, {\tt nIter}, and {\tt nFun} for all algorithms, computed over all test problems}
\begin{tabular}{|c||c|c|c||c|c|c|}
\hline
\textbf{Tested}&\multicolumn{6}{|c|}{\textbf{\phantom{joker} \; Mean Values  \, \& \, Standard Deviations}}\\
\cline{2-7} 
\textbf{ Algorithms } & \textbf{\text{ CPU(Sec) }}& \textbf{\text{ nIter }}& \textbf{\text{ nFun }}& \textbf{\text{ CPU(Sec) }}& \textbf{\text{ nIter }}& \textbf{\text{ nFun }}\\
\hline
\tt{FISTA}& ${\mathrm{10.9}}$& ${\mathrm{2527.2}}$ & ${\mathrm{2528.2 }}$& ${\mathrm{30.1}}$& ${\mathrm{2141.6}}$ & ${\mathrm{2141.6}}$\\
\hline
\tt{FPC-BB}& ${\mathrm{2.3}}$& ${\mathrm{592.7}}$ & ${\mathrm{1129.7}}$& ${\mathrm{5.7}}$& ${\mathrm{471.6}}$ & ${\mathrm{971.1}}$\\
\hline
\tt{SpaRSA}& ${\mathrm{2.5}}$& ${\mathrm{659.9}}$ & ${\mathrm{659.9}}$& ${\mathrm{6.4}}$& ${\mathrm{527.3}}$ & ${\mathrm{527.3}}$\\
\hline
\tt{TwIST}& ${\mathrm{1.3}}$& ${\mathrm{387.0}}$ & 
${\mathrm{443.4}}$& ${\mathrm{2.9}}$& 
${\mathrm{324.2 }}$ & ${\mathrm{352.5}}$\\
\hline
\tt{ISGA}& 
${\mathrm{0.9}}$& 
${\mathrm{209.7}}$ & 
${\mathrm{291.5}}$& ${\mathrm{1.6}}$& ${\mathrm{37.3}}$ & ${\mathrm{36.1}}$ \\
\hline
\tt{smISGA}& 
${\mathrm{ 0.9}}$& 
${\mathrm{216.6}}$ & 
${\mathrm{492.4}}$& ${\mathrm{1.8}}$& ${\mathrm{67.7}}$ & ${\mathrm{138.8}}$\\
\hline
\end{tabular}
\end{table}
By Figure \ref{Fig:P1} one can infer that in general, the proposed 
new solver {\tt smISGA} appears to be in total more efficient than the 
other tested solvers, as it appears on top in two out of three categories and in upper position at the third one.

Comparison to {\tt ISGA} with respect to efficiency is included in 
Table \ref{table-sb-1} which complements the performance profile
assessment. For a proper interpretation of the table, let us recall
that we deal with a statistical evaluation of randomized experiments.
Due to randomization the resulting optimization problems may be
understood as a sample of the range of problem difficulties. 
Thus our aim is
to assess how the tested methods adapt computationally to 
this variety of problems.

As we can observe, mean values of {\tt CPU(Sec)} 
and {\tt nIter} are about the same level 
for {\tt smISGA} and {\tt ISGA}. 
As for the number of function evaluations, 
the {\tt smISGA} method is slightly
more numerically expensive than the highly efficient
{\tt ISGA} and in the same regime as {\tt TwIST}.
The other methods feature much higher numbers. 

A large difference between methods can be observed 
in the standard deviations, 
indicating how the methods can cope with problems of varying difficulty. 
While both {\tt smISGA} and {\tt ISGA} methods 
stay well within a relatively small range and thus give a stable
performance independently of the difficulty level of the problem, 
standard deviations for all other methods are very high.
Which gives experimental evidence
of robustness of {\tt smISGA} and {\tt ISGA}.
\begin{figure}[b]
\floatbox[{\capbeside\thisfloatsetup{capbesideposition={right,top},capbesidewidth=6.5cm}}]{figure}[\FBwidth]
{\caption{Scatter plot to show relationships between the (logarithmic)
relative errors of {\tt ISGA} and {\tt smISGA} over all test problems. We observe that 
in a high percentage of test problems the 
relative error obtained by {\tt smISGA}
is in comparison much lower than for {\tt ISGA}.
This confirms the usefulness of the proposed semi-monotone strategy to help in convergence.
}\label{Fig:P2}}
{\includegraphics[width=5cm]{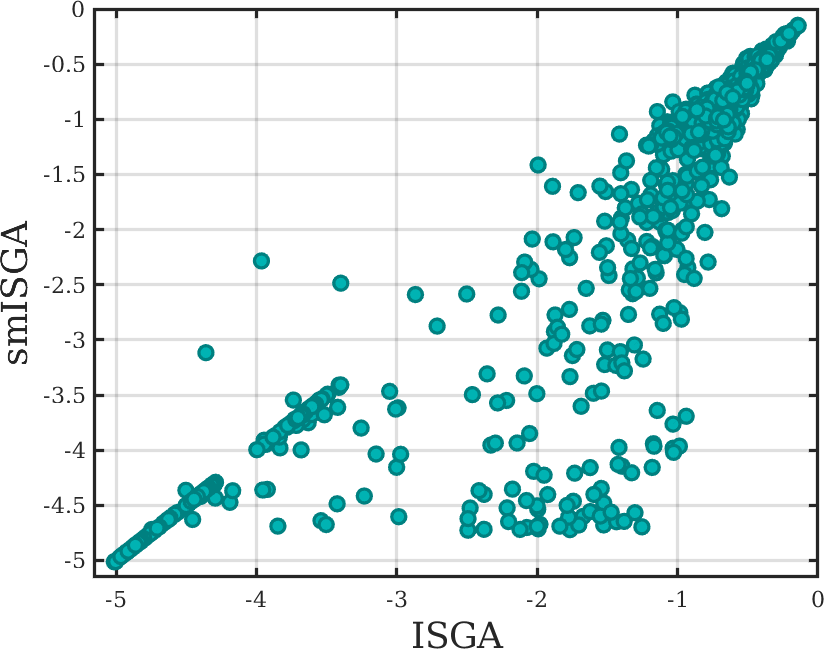}}
\end{figure}
\subsubsection{Evaluation of Quality}
We now consider the quality of results obtained
in terms of the relative error 
achieved by minimization, 
cf. Table \ref{table-sb-2}.
For proper interpretation, let us observe that
all tested methods obtain about 
the same range,
provided in terms of maximum and minimum of relative errors, whereas this
range gives an account of the
difficulty levels of the randomized problems. 

While showing a similar behaviour in terms of 
standard deviation, the mean values of both {\tt ISGA}
and {\tt smISGA} methods are higher than for the other schemes.
Where {\tt smISGA} shows at the same time
a notable improvement compared to {\tt ISGA} method, 
which is made even more apparent via Figure \ref{Fig:P2}.

\subsubsection{Conclusion on Experiments}
Complementing the study
of quality with the results on computational efficiency,
one may interpret them as follows. The
high efficiency and robustness properties 
of {\tt ISGA}/{\tt smISGA} may have in some experiments 
of medium to high difficulty the price of resulting 
in higher relative error compared to other methods. 
In turn, while the computational effort of the other tested methods
is higher in general, in some experiments of medium to high difficulty
the effort they take for minimization may grow
substantially, which may deteriorate efficiency.
Comparing {\tt smISGA} with the technically related {\tt ISGA}
method that is based on the monotone Goldstein technique, 
we achieve with the new {\tt smISGA} method a considerably
higher quality at the price of a slightly higher 
computational load in some problems.

\begin{table}[t]
\caption{\label{table-sb-2}
Average values \text{RelErr} and 
standard deviations \text{SdRelErr} 
of the relative error (left block), 
as well as the range of relative errors
given by maximum value \text{MaxRelErr} and minimum value \text{MinRelErr}
(right block), computed over all test problems}
\begin{center}
\begin{tabular}{|c||c|c||c|c|}
\hline
\textbf{Tested}&\multicolumn{4}{|c|}{\textbf{\phantom{all} Mean Values, Standard Deviations and Range \phantom{all}}} \\
\cline{2-5} 
\textbf{ Algorithms } & \textbf{ \text{RelErr } }& \textbf{ \text{SdRelErr } }& \textbf{ \text{MaxRelErr } }& \textbf{ \text{MinRelErr } }\\
\hline
\tt{FISTA}& ${\mathrm{0.2052}}$& ${\mathrm{0.2284}}$ & ${\mathrm{0.8495}}$ & ${\mathrm{0.0066}}$\\
\hline
\tt{FPC-BB}& ${\mathrm{0.2056}}$& ${\mathrm{0.2287}}$ & ${\mathrm{0.8501}}$ & ${\mathrm{0.0067}}$\\
\hline
\tt{SpaRSA}& ${\mathrm{0.2054}}$& ${\mathrm{0.2286}}$ & ${\mathrm{0.8499}}$ & ${\mathrm{0.0067}}$\\
\hline
\tt{TwIST}& ${\mathrm{0.2052}}$& ${\mathrm{0.2285}}$ & ${\mathrm{0.8498}}$ & ${\mathrm{0.0067}}$\\
\hline
\tt{ISGA}& 
${\mathrm{0.3735}}$& ${\mathrm{0.2605}}$ & ${\mathrm{0.8683}}$ & ${\mathrm{0.0067}}$\\
\hline
\tt{smISGA}& 
${\mathrm{0.3110}}$& ${\mathrm{0.2534}}$ & ${\mathrm{0.8627}}$ & ${\mathrm{0.0067}}$\\
\hline
\end{tabular}
\end{center}
\end{table}



\section{Conclusion}

We have proposed and investigated a new semi-monotone 
Goldstein technique useful for compressed sensing and related problems.
We have validated experimentally that the method is a useful
extension of a previous monotone Goldstein technique.
In future work we strive to investigate the theoretical
complexity of the approach, and we aim to tackle 
other important optimization problems as the methodology
introduced here is not limited to 
basis pursuit denoising.

\begin{credits}
\subsubsection{\ackname} We acknowledge funding by the German Aerospace Center
(DLR) as part of project AIMS: Artificial Intelligence Meets Space (grant number 50WK2270F).
\end{credits}

\end{document}